\documentclass{DMS}
\usepackage{amssymb}
\usepackage{hyperref}

\newtheorem{theorem}{Theorem}
\newtheorem{lemma}[theorem]{Lemma}
\newtheorem{corollary}[theorem]{Corollary}
\newtheorem{proposition}[theorem]{Proposition}

\numberwithin{equation}{section}

\theoremstyle{definition}

\newcommand{\R}{\mathbb{R}}
\newcommand{\eps}{\varepsilon}

\newcommand{\nx}{\nabla_x}
\newcommand{\nv}{\nabla_v}
\newcommand{\omu}{\overline{\mu}}
\newcommand{\intRd}{\int_{\R^d}}
\newcommand{\rhoF}{\rho_F}

\newcommand{\mF}{m_F}

\newcommand{\A}{\mathsf{A}}

\newcommand{\T}{\mathsf{T}}
\renewcommand{\L}{\mathsf{L}}
\newcommand{\J}{\mathsf{J}}
\renewcommand{\H}{\mathsf{H}}
\newcommand{\D}{\mathsf{D}}
\newcommand{\Rot}{\mathsf{R}}
\newcommand{\U}{\mathsf{U}}
\newcommand{\nrm}[1]{\|#1\|}
\newcommand{\scalar}[2]{\langle #1,#2\rangle}

\newcommand{\BGprofile}{\mathfrak M}

\newcommand{\be}[1]{\begin{equation}\label{#1}}
\newcommand{\ee}{\end{equation}}
\newcommand{\diff}{\sigma}

\begin{document}
\title[Hypocoercivity in kinetic models]{Hypocoercivity for linear kinetic equations conserving mass}

\author[J. Dolbeault]{Jean Dolbeault}
\address{\hspace*{-12pt}Ceremade (UMR CNRS no. 7534), Universit\'e Paris-Dauphine, Place de Lattre de Tassigny, 75775 Paris Cedex 16, France}
\curraddr{}
\email{dolbeaul@ceremade.dauphine.fr}

\author[C. Mouhot]{Cl\'ement Mouhot}
\address{\hspace*{-12pt} DMA (UMR CNRS no. 8553), \'Ecole Normale Sup\'erieure, 45, rue d'Ulm F 75230 Paris cedex 05, France}
\curraddr{University of Cambridge, DAMTP, Centre for Mathematical Sciences, Wil\-berforce Road, Cambridge CB3 0WA, UK}
\email{Clement.Mouhot@ens.fr}

\author[C. Schmeiser]{Christian Schmeiser}
\address{\hspace*{-12pt}Fakult\"at f\"ur Mathematik, Universit\"at Wien, Nordbergstra{\ss}e 15, 1090 Wien, Austria}
\curraddr{}
\email{Christian.Schmeiser@univie.ac.at}
\thanks{}





\dedicatory{}

\begin{abstract} We develop a new method for proving hypocoercivity for a large class of linear kinetic equations with only one conservation law. Local mass conservation is assumed at the level of the collision kernel, while transport involves a confining potential, so that the solution relaxes towards a unique equilibrium state. Our goal is to evaluate in an appropriately weighted $L^2$ norm the exponential rate of convergence to the equilibrium. The method covers various models, ranging from diffusive kinetic equations like Vlasov-Fokker-Planck equations, to scattering models like the linear Boltzmann equation or models with time relaxation collision kernels corresponding to polytropic Gibbs equilibria, including the case of the linear Boltzmann model. In this last case and in the case of Vlasov-Fokker-Planck equations, any linear or superlinear growth of the potential is allowed. \end{abstract}

\maketitle
\thispagestyle{empty}

\section{Method, result and consequences}\label{Sec:Intro}

\subsection{Linear kinetic equations and hypocoercivity}\label{subsec:motivation}

We consider linear kinetic equations which can be written as
\be{eq:base}
\partial_t f + \T\,f = \L\,f\,, \quad f = f(t,x,v)\,, \quad (t,x,v)\in \R^+\times\R^d\times\R^d,
\ee
and describe the evolution of a \emph{distribution function} $f$. The \emph{transport operator}
\[
\label{eq:defT}
\T := v \cdot \nx - \nx V \cdot \nv
\]
has characteristics given on the \emph{phase space} $\R^d\times\R^d$ by the flow of the Hamiltonian
\[
(x,v)\mapsto E(x,v):=\frac 12\,|v|^2 + V(x)\,.
\]
The \emph{external potential} $V=V(x)$ is a measurable function on $\R^d$. The \emph{collision operator} $\L$ is independent of \emph{time} $t$ and acts as a multiplicator in the \emph{position} variable $x$. The variable $v$ is the \emph{velocity}.

We shall consider steady states which are in the intersection of the null spaces of $\T$ and $\L$ simultaneously. We shall assume that there exists a nonnegative energy profile function $\Gamma$ such that, for each fixed value of $x$, the nullspace $\mathcal N(\L)$ of $\L$ is spanned by $F(x,v):=\Gamma(E(x,v))$, so that
\[\label{Null-L}
\mathcal N(\L) = \{ f(x,v):\,\exists\,\phi(x) \mbox{ such that } f(x,v) = \phi(x)\,F(x,v) \}\,.
\]
Functions in $\mathcal N(\L)$ are \emph{local equilibria}; they depend on $x$ and $t$. The function $F$ is a \emph{global equilibrium} or \emph{global Gibbs state}. It is independent of $t$ (stationary) and isotropic with respect to $v$. Consistently, we shall further assume that $\L$ has rotational symmetry in $v$, i.e.~$\Rot_v\,\L = \L\,\Rot_v$ for any rotation operator~$\Rot_v$ acting on the velocity space. Under the assumption that the support of $F$ is connected, the intersection of the null spaces of $\L$ an $\T$ is generated by $F$. Assume that $F$ is integrable and normalized by
\[
\iint_{\R^d\times \R^d} F\,dv\,dx = 1\,.
\]
We shall refer to this assumption as {\bf Assumption (H0)} and assume that it holds throughout the paper, although we shall not specify it explicitly when it is not useful for the understanding of our arguments. Under such a normalization condition, we shall prove that $F$ is the unique stationary distribution function. Integrability with respect to $v$ is an assumption on~$\Gamma$, whereas integrability with respect to $x$ requires a $\Gamma$-dependent growth of the external potential $V$. Such a property is a \emph{confinement condition}.

The one-dimensionality of $\mathcal N(\L)$ for fixed $x$ suggests the existence of one local (in $x$) conservation law. We shall therefore assume the \emph{ local conservation of mass,} that is
\[
\intRd \L\,f\,dv = 0\,.
\]
\emph{Global mass conservation} for solutions of \eqref{eq:base} follows:
\[
\frac d{dt}\iint_{\R^d\times \R^d}f\,dv\,dx=\iint_{\R^d\times \R^d} (\L - \T)\,f\,dv\,dx = 0\,.
\]
For an integrable initial datum
\[
f(t=0,\cdot,\cdot) = f_I\,,
\]
let $M:=\iint_{\R^d\times\R^d} f_I\,dv\,dx$, so that $MF$ is the unique global Gibbs state with mass~$M$. In this paper we investigate the asymptotic behavior of the semigroup generated by $\L-\T$. Our goal is to quantify its stability or, to be precise, to determine the rate of convergence of $f$ towards $M F$ as $t\to\infty$. Since the equation is linear, there is no restriction to study \emph{fluctuations} around a global equilibrium, that is solutions $f$ of \eqref{eq:base} which satisfy
\be{eq:masszero}
M=\iint_{\R^d\times\R^d} f(t,x,v)\,dv\,dx = \iint_{\R^d\times\R^d} f_I(x,v)\,dv\,dx = 0\,.
\ee
Notice that distribution functions are usually nonnegative, but fluctuations around an equilibrium have to change sign.

Local mass conservation for $f$ and $F$ imply the identity
\[
\intRd \L\,f\,\left( \frac{f}{F} \right)\,dv = \intRd \L\,(f-\phi\,F)\,\left(\frac{f-\phi\,F}{F} \right)\,dv
\]
for any function $\phi=\phi(x)$, thus showing that the left hand side is, at least formally, quadratic in the distance between $f$ and the kernel of $\L$. This suggests to introduce the space $L^2(d\mu)$ where the measure $d\mu$ is defined on the phase space~by
\[
d\mu=d\mu(x,v):=\frac{ dv\,dx}{F(x,v)}\,, \quad (x,v)\in \R^d\times\R^d.
\]
We shall denote by $\scalar{\cdot}{\cdot}$ the corresponding scalar product and by $\nrm\cdot$ the associated norm, so that $\scalar fg=\iint_{\R^d\times\R^d}f\,g\,d\mu$ and $\nrm f^2=\scalar{f}{f}$. The orthogonal projection~$\Pi$ on the set of local equilibria is denoted by
\[
\label{eq:defPi}
\Pi\,f := \frac{\rho_f}\rhoF\,F\,, \quad \mbox{with } \rho_f := \intRd f\,dv\,.
\]
We also assume that the collision operator is dissipative in the sense that an \emph{`H-theorem'} holds, i.e. $\scalar{\L\,f}{f}\le 0$. Since the transport operator $\T$ is skew symmetric with respect to $\scalar{\cdot}{\cdot}$, this implies the \emph{entropy inequality}
\[
\frac{1}{2}\,\frac{d}{dt}\,\|f\|^2 = \scalar{\L\,f}{f} \le 0\,.
\]
Under the normalization condition \eqref{eq:masszero}, if the entropy dissipation $-\scalar{\L\,f}{f}$ was coercive with respect to the norm $\|\cdot\|$, exponential decay to zero as $t\to\infty$ would follow. However such a coercivity property cannot hold since $\L$ vanishes on the set of local equilibria. Instead we shall assume that \emph{microscopic coercivity} holds, i.e. there exists a positive constant $\lambda_m$, such that
\[
-\scalar{\L\,f}{f}\ge \lambda_m\,\|(1-\Pi)\,f\|^2\quad\mbox{for all}\;f\in L^2(d\mu)\,.
\]
The key tool of our method is a \emph{modified entropy functional} $\H[f]$, whose square root is a norm equivalent to $\|\cdot\|$, such that
\[
\frac{d}{dt}\,\H[f] \le -\lambda\,\H[f]\,,
\]
for an explicitly computable positive constant $\lambda$. As a consequence, we find an estimate of the exponential decay rate of the semigroup. Following the vocabulary used in \cite{Mem-villani,Herau,Mouhot-Neumann}, such a strategy will be called \emph{hypocoercivity}.

\medskip In some cases, the existence of a spectral gap can be obtained by non-constructive compactness methods, see for instance \cite{Ukai-1974} in the case of the linearized Boltzmann equation on the torus. For a non-positive closed operator $\U$ with a spectral gap $\lambda >0$, it is well-known, see \cite{Pazy-1983}, that there exists a norm equivalent to the ambiant norm, for which the semigroup of $\U+\lambda$ is contractive. However this method is not constructive regarding the norm of contractivity and gives no estimate on $\lambda$. In our approach, under assumptions specifically adapted to kinetic theory, we are able to construct an explicit Hilbert norm which is equivalent to the standard norm of $L^2(d\mu)$ and to estimate $\lambda$.

Various results related to hypocoercivity have recently appeared, on large time estimates: \cite{Fellner-Neumann-Schmeiser,Caceres-Carrillo-Goudon,Desvillettes-Villani-2001,Desvillettes-Villani-2005}; based on hypoellipticity: \cite{Herau-Nier,herau2010anisotropic,hitrik2009semiclassical}; on hypocoercivity itself: \cite{Mem-villani,Herau,Mouhot-Neumann}; on applications of the so-called kinetic-fluid decomposition: \cite{Guo-2002-I,Guo-2002-II,Guo-2003-I,Guo-2003-II,Guo-2004,Guo-Strain-2004,Guo-Strain-2006,Guo-Strain-2008}; on hyperbolic estimates based on micro-macro decompositions: \cite{MR2274461,MR2044894,MR2284213,MR2264611}. Some of the results of this paper, namely Theorems~\ref{Thm:HerauImproved} and \ref{Thm:NLBGK}, have been announced in \cite{Dolbeault2009511} without complete proofs.

Our purpose is to establish, in a simplified framework, sufficient conditions for proving hypocoercivity for a large class of linear kinetic models confined by an external potential, \emph{without assuming regularity on the initial datum} and \emph{valid for hypoelliptic kinetic Fokker-Planck equations as well as singularity preserving collisional kinetic equations}. This is the main difference with hypoelliptic methods. The method also makes use of a micro-macro decomposition. Accordingly we shall split our assumptions into two main requirements: \emph{microscopic coercivity} as introduced above, and a \emph{macroscopic coercivity} assumption, which is a spectral gap-like inequality for the operator obtained when taking an appropriate macroscopic diffusion limit that we shall now describe.

\subsection{Formal macroscopic limit}\label{subsec:difflimit}

As a motivation for the \emph{macroscopic coercivity} assumption, we recall, at a formal level, the macroscopic diffusion limit procedure, which can be seen as intermediate asymptotics governing the long time behaviour of solutions. On a large time scale, it heuristically models how local equilibria relax towards the global Gibbs state. Since the macroscopic flux of the equilibrium distribution vanishes, i.e. $\intRd v\,F\,dv = 0$, the appropriate macroscopic rescaling of the solution of \eqref{eq:base} is given by
\[
f^\eps (t,x,v) = f( \eps^2\,t, \eps\,x , v)\,,\quad 0 < \eps \ll 1\,,
\]
which is known as the \emph{parabolic rescaling}. Assuming that the potential $V$ is rescaled accordingly, we obtain the singular limit problem
\[\label{eq:diffrescale}
\eps^2\,\partial_t f^\eps + \eps\,\T\,f^\eps = \L\,f^\eps\,,
\]
as $\eps \to 0$. The assumption $\lim_{\eps\to 0} f^\eps = f^0$ leads to $\L\,f^0 = 0$ and, thus, $f^0 = \Pi\,f^0$. The identities $\Pi\,\L = \L\,\Pi = \Pi\,\T\,\Pi = 0$ imply the relations
\[\label{R-equ}
\eps\,\partial_t f^\eps + \T\,f^\eps = \L\,R^\eps\,,\quad \partial_t \Pi\,f^\eps + \Pi\,\T\,R^\eps = 0\,,
\]
with $R^\eps := \tfrac 1\eps\,(1-\Pi)\,f^\eps$. Assuming formally that $\lim_{\eps\to 0} R^\eps = R^0$, the first equation can be solved for $\eps=0$ with respect to $R^0$, giving
\[
R^0 = \J\,\T\,f^0 = \J\,\T\,\Pi\,f^0\,,
\]
where $\J$ denotes the inverse of the restriction of $\L$ to the orthogonal complement of its null space. Note that the inhomogeneity $\T\,\Pi\,f^0$ satisfies the solvability condition $\Pi\,\T\,\Pi\,f^0 = 0$. The second equation becomes
\be{eq:abstractlimit}
\partial_t\,\Pi\,f^0 = (\T\,\Pi)^*\J\,(\T\,\Pi)\,f^0\,,
\ee
where the superscript $*$ denotes the adjoint operator with respect to $\scalar{\cdot}{\cdot}$, and the skew symmetry of $\T$ has been used. A straightforward computation shows that this is equivalent to a drift-diffusion equation for the macroscopic density $\rho^0 = \rho_{f^0}$:
\be{eq:evolrhofinal}
\partial_t \rho^0 = \nx\cdot \Big[\rhoF\,\diff\,\nx\big(\tfrac{\rho^0}{\rhoF}\big) \Big] = \nx\cdot \Big[\nx(\diff\,\rho^0) + \gamma\,\rho^0\,\nx V \Big]\,.
\ee
Here $\diff$ is scalar due to the rotational symmetry of $\L$,
\[
\rhoF\,\diff = - \frac{1}{d} \intRd v\cdot \J\,(v\,F)\,dv\quad\mbox{and}\quad\gamma\,\nx V = -\,\frac 1\rhoF\,\nx(\rhoF\, \diff)\,.
\]
The operator $\J$ being negative definite on $(1-\Pi)\,L^2(d\mu)$, $\diff(x) > 0$ for all $x\in\R^d$. In the two following important cases, the macroscopic transport coefficients $\gamma$ and~$\diff$ have particularly simple expressions.

\medskip\noindent{\bf Case (C1)}. When $\Gamma(s)=e^{-s}$, the global Gibbs state is a \emph{Maxwellian}, or Gaussian function, which factorizes as
\[
F(x,v) = \rhoF(x)\,\BGprofile(v)\,, \quad\mbox{with}\quad \rhoF=\frac{e^{-V}}{\intRd e^{-V}\,dx}\quad\mbox{and}\quad\BGprofile(v) = \frac{e^{-|v|^2/2}}{(2\,\pi)^{d/2}}\,.
\]
Notice that the separation of position and velocity variables is a characteristic property of Maxwellian functions. Both coefficients $\gamma$ and $\diff$ are constant, equal to $\frac 1d\intRd v\cdot \J\,(v\,\BGprofile)\,dv$ and $\rho^0$ solves the Fokker-Planck equation
\[
\partial_t \rho^0 = \diff\,\nx\cdot \left(\nx\rho^0 + \rho^0 \,\nx V \right)\,.
\]

\medskip\noindent{\bf Case (C2)}. The collision operator $\L$ is, for fixed $x$, a \emph{time-relaxation operator} onto~span$\{F\}$, i.e.,
\[
\label{eq:relax}
\L = \Pi - 1\,.
\]
In this case $\J =-\mbox{Id}$ holds, so that
\[
\rhoF\,\diff = \mF := \frac 1d \intRd |v|^2\,F\,dv\,,
\]
and, since $\nx(\rhoF\,\diff)=-\,\rhoF\, \nx V$ because
\[
\frac 1d \intRd |v|^2\,\Gamma'(E(x,v))\,dv = \frac 1d \intRd v\cdot\nv F\,dv = -\,\rhoF\,,
\]
the macroscopic limit equation reads
\[
\partial_t \rho^0 = \nx \cdot \left( \nx(\diff\,\rho^0)+\rho^0\,\nx V \right)\,.
\]

\medskip The intersection of both cases, (C1) and (C2) i.e., $\L = \Pi - 1$ with $\Pi\,f = \rho_f\,\BGprofile$, gives $\diff = \gamma = 1$. This is the linear BGK case, which has been considered in \cite{Dolbeault2009511}. In both cases, \eqref{eq:abstractlimit} can be rewritten as
\[
\partial_t \Pi\,f^0 = -\,\sigma_0\,(\T\,\Pi)^* (\T\,\Pi)\,f^0
\]
for some positive constant $\sigma_0$, with $\sigma_0\equiv\sigma$ in Case (C1) and $\sigma_0 = 1$ in Case~(C2). With Assumption \eqref{eq:masszero} on the initial data, we expect decay to zero of the solution. Under a \emph{macroscopic coercivity} assumption, namely {\bf (H2)} (see below), which is equivalent to a Poincar\'e inequality (see Lemma~\ref{eq:macroSG}),the decay of $\Pi\,f^0$ is exponential.

\subsection{Method and main result in an abstract setting}\label{sec-abstract}

We start with the basic assumption that $\L$ and $\T$ are closed linear operators on an Hilbert space $\mathcal H$, such that $\L-\T$ generates the strongly continuous semigroup $e^{(\L-\T)\,t}$ on $\mathcal H$. The orthogonal projection on the null space $\mathcal N(\L)$ of $\L$ is denoted by $\Pi$ and $\mathcal D(\L)$ is the domain of~$\L$. We assume that the restriction of $\L$ to $\mathcal N(\L)^\bot$ is coercive. More precisely, our first assumption is:
\par\medskip\noindent{\bf Assumption (H1)} (microscopic coercivity): \emph{The operator $\L$ is symmetric and there exists $\lambda_m > 0$ such that
\[
-\scalar{\L\,f}{f} \ge \lambda_m\,\|(I-\Pi)\,f\|^2\quad\mbox{for all}\;f\in \mathcal D(\L)\,.
\]}
Motivated by the results of Section \ref{subsec:difflimit}, coercivity of the transport operator is required, when acting on $\mathcal N(\L)$:
\par\medskip\noindent{\bf Assumption (H2)} (macroscopic coercivity): \emph{The operator $\T$ is skew symmetric and there exists $\lambda_M > 0$ such that
\[
\| \T\,\Pi\,f\|^2 \ge \lambda_M\,\|\Pi\,f\|^2\quad\mbox{for all}\;f\in \mathcal H\;\mbox{such that}\;\Pi\,f\in \mathcal D(\T)\,.
\]}

Inspired by \cite{Herau}, we introduce the \emph{modified entropy}
\[\label{Lyapunov}
\H[f] := \frac 12\,\|f\|^2 + \eps\,\scalar{\A\,f}{f}\,, \quad\mbox{with }\A := \big(1 + (\T\,\Pi)^*(\T\,\Pi) \big)^{-1} (\T\,\Pi)^*\,.
\]
The constant $\eps>0$ will be chosen below. A straightforward computation for a solution $f$ of \eqref{eq:base}, now considered as an abstract ODE, gives
\[
\frac d{dt}\,\H[f]= -\D[f]
\]
where the \emph{dissipation of entropy functional} is given by
\[
\D[f]:= -\scalar{\L\,f}{f} + \eps\,\scalar{\A\,\T\,\Pi\,f}f + \eps\,\scalar{\A\,\T\,(1-\Pi)\,f}f - \eps\,\scalar{\T\,\A\,f}f - \eps\,\scalar{\A\,\L\,f}{f}\,.
\]
By {\bf (H1)}, {\bf (H2)}, and by $\A\,\T\,\Pi = (1 + (\T\,\Pi)^*(\T\,\Pi))^{-1} (\T\,\Pi)^*(\T\,\Pi)$, the sum of the first two terms in $\D[f]$ is coercive:
\[
-\scalar{\L\,f}{f}+\eps\,\scalar{\A\,\T\,\Pi\,f}f \ge \min\Big\{ \lambda_m, \frac{\eps\,\lambda_M}{1+\lambda_M} \Big\}\,\|f\|^2
\]
For the completion of our program, we need to show that $\H[f]$ is equivalent to $\|f\|^2$ and to control the last three terms of $\D[f]$. Part of this can be carried out at the abstract level under the following additional assumption:
\par\medskip\noindent{\bf Assumption (H3)}:
\[
\Pi\,\T\,\Pi = 0\,.
\]

\begin{lemma}
\label{lemma1}
Let Assumptions {\bf (H1)--(H3)} hold. Then the operators $\A$ and $\T\,\A$ are bounded, and for all $f\in \mathcal H$,
\be{ExplicitBound}
\nrm{\A\,f} \le \frac 12\,\nrm{(1-\Pi)\,f} \quad\mbox{and}\quad \nrm{\T\,\A\,f} \le \nrm{(1-\Pi)\,f}\,.
\ee
\end{lemma}
\begin{proof} The equation $\A\,f = g$ is equivalent to
\[\label{Af=g}
(\T\,\Pi)^* f = g + (\T\,\Pi)^*(\T\,\Pi)\,g\,.
\]
Writing this as $g = \Pi\,\T^2\,\Pi\,g - \Pi\,\T\,f$ proves $\A=\Pi\,\A$ and, thus, $\T\,\A f = \T\,\Pi\,g$. Taking the scalar product of the above equality with $g$ and using {\bf (H3)}, we get
\begin{eqnarray*}
\|g\|^2 + \|\T\,\Pi\,g\|^2 &=& \scalar{f}{\T\,\Pi\,g} = \scalar{(1-\Pi)\,f}{\T\,\Pi\,g} \\
&\le& \nrm{(1-\Pi)\,f}\;\nrm{\T\,\Pi\,g} \le \frac 14\,\nrm{(1-\Pi)\,f}^2 + \nrm{\T\,\Pi\,g}^2\,,
\end{eqnarray*}
which completes the proof.\end{proof}

The boundedness of the remaining terms in $\D[f]$ has to be proven case by case. We shall therefore assume it in the abstract setting.
\par\medskip\noindent{\bf Assumption (H4)} (Boundedness of auxiliary operators): \emph{The operators $\A\,\T\,(1-\Pi)$ and $\A\,\L$ are bounded, and there exists a constant $C_M>0$ such that, for all~$f\in\mathcal H$,
\[
\|\A\,\T\,(1-\Pi)\,f\| + \|\A\,\L\,f\| \le C_M\,\|(1-\Pi)\,f\|\,.
\]}

\begin{theorem}
\label{theo:abstract}
Let Assumptions {\bf (H1)--(H4)} hold. Then there exist positive constants $\lambda$ and $C$, which are explicitly computable in terms of $\lambda_m$, $\lambda_M$, and $C_M$, such that, for any initial datum $f_I\in \mathcal H$,
\[
\label{eq:decH}
\big\|e^{t\,(\L-\T)} f_I\big\| \le C\,e^{-\lambda\,t}\,\|f_I\|\,,\quad\forall\,t\ge0\,.
\]
\end{theorem}

\begin{proof} The first inequality in \eqref{ExplicitBound} implies
\be{equivalence}
\frac 12\,(1-\eps)\,\|f\|^2 \le \H[f] \le \frac 12\,(1+\eps)\,\,\|f\|^2\,.
\ee
For any $\eps\in(0,1)$, $\H[f]$ is equivalent to $\|f\|^2$. The second inequality in \eqref{ExplicitBound} and {\bf (H1)--(H4)} imply
\begin{eqnarray*}
\D[f] &\kern -3pt\ge&\kern -3pt \lambda_m\,\|(1-\Pi)\,f\|^2 + \tfrac{\eps\,\lambda_M}{1+\lambda_M}\,\|\Pi\,f\|^2 - \eps\,(1+C_M)\,\|(1-\Pi)\,f\|\,\|f\| \\
&\kern -3pt\ge&\kern -3pt \left[\lambda_m - \eps\,(1+C_M)(1 + \tfrac{1}{2\,\delta})\right] \|(1-\Pi)\,f\|^2 + \eps\left[\tfrac{\lambda_M}{1+\lambda_M} - (1+C_M)\tfrac{\delta}{2}\right] \|\Pi\,f\|^2
\end{eqnarray*}
for an arbitrary positive $\delta$. By choosing first $\delta$ and then $\eps$ small enough, a positive constant $\kappa$ can be found, such that $\D[f]\ge \kappa\,\|f\|^2$. Using \eqref{equivalence}, this implies
\[
\frac{d}{dt}\,\H[f] \le -\frac{2\,\kappa}{1+\eps}\,\H[f]\,,
\]
for $f = e^{t\,(\L-\T)}f_I$, completing the proof with $\lambda = \kappa/(1+\eps)$ and $C = \sqrt{\frac{1+\eps}{1-\eps}}$.
\end{proof}

Let us conclude this abstract approach by some comments. First of all, our proof is constructive: $\H$ is an explicit Lyapunov functional and $\lambda$ can be computed. The work of F. H\'erau has been a crucial source of inspiration for our method. In \cite{Herau}, he deals with the linear time relaxation collision kernel corresponding to Maxwellian Gibbs states, in case of a confining potential $V$ growing at most quadratically at infinity and such that the associated Witten Laplacian satisfies a spectral gap inequality. In our approach, we are able to relax some of these assumptions. See Theorem~\ref{Thm:HerauImproved}.

Our results apply to various Fokker-Planck and Boltzmann models. We shall compare applications of Theorem~\ref{theo:abstract} to previous results in Section 3. Only \cite{Mouhot-Neumann} and \cite{Mem-villani} deal with abstract results like the ones of Theorem~\ref{theo:abstract}. Ours are more general than the ones of \cite{Mouhot-Neumann} since we deal with a general confining potential. In~\cite{Mouhot-Neumann}, the problem is indeed set on a torus, a setting to which our method can be adapted without any difficulty. It is also more general than in \cite{Mem-villani} since we deal not only with Fokker-Planck type operators, or operators in {\rm H\"ormander form} in the words of \cite{Mem-villani}, but also with non-local integral collision operators, like in \cite{Mouhot-Neumann}. Last but not least, our results are also stronger than those in \cite{Mouhot-Neumann} and \cite{Mem-villani} in the sense that we construct a \emph{zeroth order norm of hypocoercivity}, which is equivalent to $L^2$ and not $H^k$, for some $k \ge 1$. However, our results are weaker than those in \cite{Mouhot-Neumann} at least in one aspect: we only deal with models with $1$-dimensional space of collision invariants, whereas, in~\cite{Mouhot-Neumann}, any finite dimension is allowed. In principle, our approach can be extended to such a situation, which is the purpose of a current research project \cite{DMS-2part}.

\subsection{Hypocoercivity for a toy problem}\label{subsec:toy}

To illustrate the fact that our formal setting applies to other models than the kinetic equations of Section~\ref {subsec:motivation}, we introduce the following toy model, which captures very well the essential features of our hypocoercive approach. We consider a one-dimensional Cattaneo model introduced in \cite{MR0032898}, which can be written as a kinetic model with only two velocities $v = \pm 1$, and where $\L$ describes a switching process between the two velocities without preference for one of them. As a further simplification we replace the confining potential by a periodicity assumption, where $x$ varies in a one-dimensional torus. The model equations are
\[
\partial_t f^\pm \pm \partial_x f^\pm = \pm\,\frac12\,(f^- - f^+)\,,
\]
for the distributions $f^\pm(t,x)$ of right- and left-moving particles, periodic in $x$ with period $2\,\pi$.

The interest of such a model is that it gives an application of our hypocoercivity method in a discrete setting, or even for a finite dimensional ODE version of it, if we truncate the Fourier sum in the $x$ variable and keep only a finite number of terms.

Initial value problems can be solved explicitly by Fourier decomposition. Introducing the total density $\rho = f^+ + f^-$, the total flux $j = f^+ - f^-$, and their Fourier representations
\[
\rho(t,x) = \sum_{k\in\mathbb{Z}} \rho_k(t)\,e^{ikx}\,,\quad j(t,x) = \sum_{k\in\mathbb{Z}} j_k(t)\,e^{ikx}\,,
\]
leads to real ODE systems for $U_k={u_k \choose v_k} := {{\rm Re}(\rho_k) \choose {\rm Im}(j_k)}$ and $\widetilde U_k:= {{\rm Im}(\rho_k) \choose {\rm Re}(j_k)}$:
\be{eq:toy}
\frac{dU_k}{dt}+\T_k\,U_k=\L\,U_k\,,
\ee
where the skew symmetric matrix $\T_k:=\left(\begin{smallmatrix}0&-k\cr k&0\end{smallmatrix}\right)$ represents the transport operator, $\L:=\left(\begin{smallmatrix}0&0\cr 0&-1\end{smallmatrix}\right)$ represents the collision operator acting only on the microscopic component $j$, and $\widetilde U_k$ solves an analogous system with $\T_k$ replaced by $-\T_k$. Eq. \eqref{eq:toy} is linear, and it is elementary to check that the eigenvalues of $\L-\T$ are given by $\lambda_{0,\pm}:=0,-1$ and $\lambda_{k,\pm}:=(-1\pm i\sqrt{4\,k^2-1})/2$ if $k\ne 0$. All solutions converge to an eigenstate of the zero eigenvalue: $U_0 = (\rho_0,0)$ and $U_k = 0$ for $k\ne 0$. The convergence is exponential with its speed determined by the spectral gap $1/2$.

For $k\ne 0$, we can compute the entropy dissipation as
\[
\frac d{dt}\left(\tfrac 12\,|U_k(t)|^2\right)=-|v_k(t)|^2\,,
\]
so that it is clear that no exponential decay directly follows, since the right hand side is not coercive and there is an unbounded increasing sequence $(t_n)_{n\in\mathbb N}$ such that $v_k(t_n)=0$. Note that microscopic coercivity holds with $\lambda_m = 1$.

With $\Pi=\left(\begin{smallmatrix}1&0\cr 0&0\end{smallmatrix}\right)$ and $\J =-\mbox{Id}$, we find that $(\T\,\Pi)^*\J\,(\T\,\Pi)=\left(\begin{smallmatrix}-k^2&0\cr 0&0\end{smallmatrix}\right)$, thus giving for the macroscopic diffusion limit $du_k^0/dt = -k^2 u_k^0$, and showing also that macroscopic coercivity holds with $\lambda_M = 1$. According to the strategy of the Section~\ref{sec-abstract}, for $k\ne 0$ we introduce the modified entropies
\[
\H_k(t)=\tfrac 12\,|U_k(t)|^2+\eps\,\frac k{1+k^2}\,u_k(t)\,v_k(t)\,,\quad t\ge 0\,.
\]
Observing that, for $k\ge 1$,
\begin{multline*}
\tfrac 12 \left(1-\eps\,\tfrac k{1+k^2}\right)|U_k|^2 \le \tfrac 12 \left(1-\eps\,\tfrac k{1+k^2}\right)|U_k|^2+\tfrac\eps{2}\,\tfrac k{1+k^2}\,|u_k+v_k|^2=\H_k\,,\\
\H_k= \tfrac 12 \left(1+\eps\,\tfrac k{1+k^2}\right)\,|U_k|^2 - \tfrac\eps{2}\,\frac k{1+k^2}\,|u_k-v_k|^2 \le\tfrac 12 \left(1+\eps\,\tfrac k{1+k^2}\right)|U_k|^2\,,
\end{multline*}
using $\sup_{k\ge 1}\tfrac k{1+k^2}=\tfrac 12\le 1$, and performing a similar computation for $k\le -1$, we finally get
\[
\tfrac 12\,(1-\eps)\,|U_k|^2\le \H_k\le\tfrac 12\,(1+\eps)\,|U_k|^2.
\]
Hence, for any $\eps\in(0,1)$, $\H_k(t)$ decays exponentially if and only if $|U_k(t)|^2$ decays exponentially as well. Obviously, we have
\[
\frac 12\le\frac{k^2}{1+k^2}\le 1\,,
\]
which makes it easy to compare $\H_k$ with $\frac d{dt}\,\H_k$, given by
\begin{multline*}
\frac d{dt}\,\H_k=-\eps\,\frac{k^2}{1+k^2}\,u_k^2-\left(1-\eps\,\frac {k^2}{1+k^2}\,\right)v_k^2- \eps\,\frac k{1+k^2}\,u_k\,v_k\\
\le-\frac\eps 2\,u_k^2-\left(1-\eps\right)\,v_k^2+ \frac\eps{2}\,|u_k|\,|v_k| \le-\frac{\eps}{2}\,(1-\lambda^2)\,u_k^2 - \left(1-\eps-\tfrac\eps{8\,\lambda^2}\right)v_k^2
\end{multline*}
for any $\lambda\in(0,1)$. If $\eps\in\big(0,\tfrac{8\,\lambda^2}{8\,\lambda^2+1}\big)$, the coercivity constant
\[
\kappa:=\min\left\{\tfrac{\eps}{2}\,(1-\lambda^2),\,1-\eps-\tfrac\eps{8\,\lambda^2}\right\}
\]
is positive and
\[
\frac d{dt}\,\H_k\le-\kappa\,|U_k|^2\le - \frac{2\,\kappa}{1+\eps}\,\H_k\;.
\]
This implies $|U_k(t)|$ decays like $e^{-\kappa\,t/(1+\eps)}$. We may observe that
\[
\frac{\kappa}{1+\eps} < \frac{\min\{\frac\eps{2},1-\eps\}}{1+\eps} \le \frac 15\,,
\]
thus showing that the method is not optimal, in the sense that it does not give the exact decay rate, $1/2$, even when refining the above estimates and computing $\kappa$ for each $k$.

\subsection{Application to kinetic equations}

Let us apply the abstract procedure of Section \ref{sec-abstract} to the setting of Section \ref{subsec:motivation}. Thus, we
set
\[
\T\,f = v\cdot\nx f - \nx V\cdot\nv f\,,\quad \Pi\,f = \rho_f\,\frac{F}{\rhoF}\,,
\]
where the potential $V$ is given as well as the energy profile $\Gamma$. We recall that the unique global equilibrium is $F(x,v) = \Gamma(|v^2|/2 + V(x))$, $x,v\in\R^d$. For such an equilibrium distribution, define the velocity moments up to the fourth order by
\[
\rhoF:= \intRd F\,dv\,,\quad\mF:= \frac{1}{d} \intRd |v |^2 F\,dv\,,\quad M_F:= \intRd |v|^4 F\,dv
\]
and assume that they are measurable functions of $x$. We consider the Hilbert space $\mathcal H = \{ f\in L^2(d\mu):\, \iint_{\R^d\times\R^d} f\,dv\,dx = 0\}$, with $d\mu(x,v) = dx\,dv/F(x,v)$. The collision operator $\L$ remains unspecified at the moment, so that we shall defer the discussion of the microscopic coercivity for a while. A simple computation with $u=\rho_f/\rhoF$ shows that the macroscopic coercivity assumption is equivalent to a weighted Poincar\'e inequality:
\begin{lemma} Assumption {\bf (H2)} holds if and only if
\be{eq:macroSG}
\intRd \left| \nx u \right|^2 \mF\,dx \ge \lambda_M \intRd u^2\,\rhoF\,dx
\ee
for any $u\in L^2(\rhoF\,dx)$ with $\nx u \in L^2(\mF\,dx)$ such that $\intRd u\,\rhoF\,dx = 0$.
\end{lemma}

In case of kinetic equations, Assumption {\bf (H3)} is a consequence of the computation
\[
\T\,\Pi\,f = F\,v\cdot\nx u_f\,,
\]
with $u_f:=\rho_f/\rhoF$, and of the observation that the right hand side is an odd function of $v$, whose mean value is zero. In other words: The macroscopic flux of the equilibrium distributions vanishes.

Concerning the Assumption {\bf (H4)}, we remark that boundedness of $\A\,\L$ is possible even for unbounded collision operators $\L$ (see Section~\ref {sec:regularity}). The boundedness assumption on $\A\,\T\,(1-\Pi)$ can be interpreted as an elliptic regularity result for:
\be{eq:u_g}
u - \frac{1}{\rhoF} \nx\cdot (\mF\,\nx u) = w\,.
\ee
\begin{lemma} \label{lemma2} If there exists a positive constant $C$ such that
\be{regularity-est}
\|\nx^2u\|_{L^2(M_F\,dx)} \le C\,\|w\|_{L^2(\rhoF\,dx)}
\ee
for any $w\in L^2(\rhoF\,dx)$ and for any solution $u\in L^2(\rhoF\,dx)$ with $\nx u\in L^2(\mF\,dx)$, then the operator $\A\,\T\,(1-\Pi)$ is bounded on $\mathcal H$.
\end{lemma}
\begin{proof} The operator $\A\,\T\,(1-\Pi)$ is bounded if and only if its adjoint
\[
[\A\,\T\,(1-\Pi)]^* = -(1-\Pi)\,\T^2\,\Pi\,[1+(\T\,\Pi)^*(\T\,\Pi)]^{-1}
\]
is bounded. If $g=[1+(\T\,\Pi)^*(\T\,\Pi)]^{-1}f$, then
\[
[\A\,\T\,(1-\Pi)]^*f = -(1-\Pi)\,\T^2\,\Pi\,g \quad\mbox{with}\quad g + (\T\,\Pi)^*(\T\,\Pi)\,g = f\,,
\]
where the latter implies \eqref{eq:u_g} for $u=u_g = \rho_g/\rhoF$ and $w=u_f = \rho_f/\rhoF$. Then
\[
\T^2\,\Pi\,g = F\,v\cdot\nx (v\cdot\nx u) + F\,\nx V\cdot\nx u
\]
results in
\[
[\A\,\T\,(1-\Pi)\,]^*f = -(1-\Pi)\,[F\,v\cdot\nx (v\cdot\nx u)]\,.
\]
This implies that for some positive constant $C$, we have
\[
\|[\A\,\T\,(1-\Pi)]^*f\| \le \|F\,v\cdot\nx (v\cdot\nx u)\| \le c\,\|\nx^2 u\|_{L^2(M_F\,dx)}\,,
\]
which completes the proof using $\|u_f\|_{L^2(\rhoF\,dx)} = \|\Pi\,f\|$. \end{proof}

\section{A framework for the elliptic regularity estimate}\label{sec:regularity}

Our goal is to give conditions on $V$ which are sufficient to establish the existence of a positive constant $C$ as in Lemma~\ref{lemma2}. ÊIn the applications considered below, the combination of weights $M_F\,\rhoF/\mF^2$ is constant. This motivates the notations
\[
w_0^2 := \rhoF\,,\quad w_i^2 := {\textstyle\big(\frac{\mF}\rhoF \big)}^i w_0^2\quad\mbox{with}\quad i=1,\,2\,,\quad \|u\|_i := \|u\,w_i\|_{L^2(\R^d)}\,.
\]
With $\rho = u\rhoF$, the Poincar\'e inequality in {\bf (H3)} can then be rewritten as
\be{Poincare1}
\|\nx u\|_1^2 \ge \lambda_M\,\| u\|_0^2
\ee
under the zero average condition $\intRd u\,\rhoF\,dx=0$, and the desired estimate \eqref{regularity-est}~is
\be{u-est1}
\|\nx^2 u\|_2 \le C\,\|u_f\|_0
\ee
for the solution of
\be{u-equ1}
w_0^2\,u - \nx\cdot (w_1^2\,\nx u) = w_0^2\,u_f\,.
\ee
Roughly speaking we just have to prove $(L^2\to H^2)$-regularization for a second order elliptic equation. However, different norms have to be taken into account. The result can only be shown under certain assumptions on the weights, which will later be translated into assumptions on the confining potential:
\be{V-ass1}
\exists\,c_1 > 0\,,\; c_2\in[0,1)\;\mbox{such that}\; -\,w_1^2\,\Delta_x (\log w_1) \le c_1\,w_0^2 + c_2\,\left|\nx w_1 \right|^2 ,
\ee
\be{V-ass2}
\exists\,c_3 > 0\;\mbox{such that}\; \frac{w_1}{w_0}\,|\nx\,W| \le c_3\,\left(1 + \frac{|\nx w_1|}{w_0} \right) \;\mbox{with }\; W := \sqrt{1 +\,\Big| \frac{\nx w_1}{w_0} \Big|^2}\,,
\ee
\be{V-ass3}
\exists\,c_4 > 0\;\mbox{such that}\; \left| \nx \Big(\frac{w_1}{w_0} \Big) \right| \le c_4\,\frac{|\nx w_1|}{w_0}\,,
\ee
\be{H7}
\|W\|_0^2=\intRd W^2\,\rhoF\,dx<\infty\,.
\ee
Note that a condition on the third weight function $w_2$ could be deduced from \eqref{V-ass1}-\eqref{V-ass2}-\eqref{V-ass3} since any two of the weight functions determine the last one. The goal of this section is to prove the following $H^2$-regularity estimate.
\begin{proposition}\label{Lem:BK}
Let \eqref{Poincare1}, \eqref{V-ass1}, \eqref{V-ass2}, \eqref{V-ass3}, and \eqref{H7} hold. Then the solution~$u$ of \eqref{u-equ1} satisfies \eqref{u-est1}.
\end{proposition}
By Lemma~\ref{lemma2}, this shows that the operator $\A\,\T\,(1-\Pi)$ is bounded.

\subsection{Improved Poincar\'e inequalities}

We start with an improvement of the Poincar\'e inequality \eqref{Poincare1}.
\begin{lemma} \label{lem-PMi}
Let \eqref{Poincare1} and \eqref{V-ass1} hold. There exists $\kappa>0$ such that
\be{Poincare1-improved}
\| \nx u \|_1^2 \ge \kappa\,\Big\| u\,\frac{\nx w_1}{w_0} \Big\|_0^2\quad\mbox{for any}\; u\in L^2(\rhoF\,dx)\;\mbox{with}\;\intRd u\,\rhoF\,dx=0\,.
\ee
\end{lemma}
\begin{proof} With the identity $w_1\nx u = \nx (w_1\,u) - u\,\nx w_1$, the inequality
\begin{eqnarray*}
\| \nx u \|_1^2 &\ge& \intRd u^2\,|\nx w_1|^2\,dx - 2 \intRd u\,\nx w_1 \cdot \nx (u\,w_1)\,dx \\
&=& \Big\| u\,\frac{\nx w_1}{w_0} \Big\|_0^2 + \intRd u^2\,w_1^2\,\Delta_x (\log w_1)\,dx
\end{eqnarray*}
is easily derived. Now \eqref{Poincare1} and \eqref{V-ass1} imply
\[
\| \nx u \|_1^2 \ge (1-c_2)\,\Big\| u\,\frac{\nx w_1}{w_0} \Big\|_0^2 - \frac{c_1}{\lambda_M}\,\| \nx u \|_1^2\,.
\]
This completes the proof with $\kappa = \lambda_M\,(1-c_2)/(\lambda_M+c_1)$. \end{proof}

\begin{lemma}\label{lem-PMhati}
Let \eqref{Poincare1}, \eqref{V-ass1}, \eqref{V-ass2} and \eqref{H7} hold. There exists $\kappa'>0$ such that
\[
\label{Poincare1-improved2}
\| W\,\nx u \|_1^2 \ge \kappa'\,\Big\| W\,u\,\frac{\nx w_1}{w_0} \Big\|_0^2\quad\mbox{for any}\; u\in L^2(\rhoF\,dx)\;\mbox{with}\;\intRd u\,\rhoF\,dx=0\,.
\]
\end{lemma}
\begin{proof}
We apply \eqref{Poincare1-improved} with $u$ replaced by $(u\,W-\overline u)$ with $\overline u:=\intRd u\,W\,\rhoF\,dx$. We recall that $\intRd\rhoF\,dx=\iint_{\R^d\times\R^d}F\,dv\,dx=1$. We thus obtain
\[
\kappa\,\big\|(u\,W-\overline u)\,W\big\|_0^2\le\|\nx(W\,u)\|_1^2\,.
\]
By expanding the left hand side, we get
\[
\kappa\,\|W^2\,u\|_0^2\le\|\nx(W\,u)\|_1^2+2\,\kappa\intRd W^3\,u\, w_0^2\,dx\;\|W\|_0\,\|u\|_0\,.
\]
Using
\[
\intRd W^2\,u\,W\,\rhoF\,dx\leq \frac 1{2\,a}\,\|W^2\,u\|_0^2+\frac a2\,\|W\|_0^2
\]
with $a:=2\,\|W\|_0\,\|u\|_0$, we obtain
\begin{eqnarray*}
2\,\kappa\intRd W^3\,u\, w_0^2\,dx\;\|W\|_0\,\|u\|_0&\leq& \kappa\,a\left(\frac 1{2\,a}\,\|W^2\,u\|_0^2+\frac a2\,\|W\|_0^2
\right)\\
&&=\frac \kappa 2\,\|W^2\,u\|_0^2+2\,\kappa\,\|W\|_0^4\,\|u\|_0^2\,.
\end{eqnarray*}
On the other hand, we can also expand the square in $\|\nx(W\,u)\|_1^2$
\[
\|\nx(W\,u)\|_1^2\le 2\,\|W\,\nx u\|_1^2 + 2\,\Big\| u\,\nx W\,\frac{w_1}{w_0}\,\Big\|_0^2\,,
\]
and by \eqref{V-ass2}, get
\[
\Big\| u\,\nx W\,\frac{w_1}{w_0}\,\Big\|_0^2\le c_3^2\,\Big\| u\,\left({\textstyle 1 + \frac{|\nx w_1|}{w_0} } \right)\Big\|_0^2\le 2\,c_3^2\,\|W\,u\|_0^2\,.
\]
Collecting all terms, we finally end up with
\[
\kappa\,\|W^2\,u\|_0^2\le 2\,\|W\,\nx u\|_1^2+4\,c_3^2\,\|W\,u\|_0^2+\frac\kappa 2\,\|W^2\,u\|_0^2 +2\,\kappa\,\|W\|_0^4\,\|u\|_0^2\,.
\]
which, using \eqref{Poincare1}, \eqref{Poincare1-improved} and $W\ge 1$ establishes the inequality:
\[
\frac\kappa 2\,\|W^2\,u\|_0^2\le \Big(2+4\,\frac{c_3^2}\kappa+2\,\frac{\kappa}{\lambda_M}\,\|W\|_0^4\Big)\,\|W\,\nx u\|_1^2\,.
\]
\end{proof}

\subsection{The regularity estimate}

Now we start working on Equation \eqref{u-equ1}. The standard energy estimate gives
\begin{equation}\label {BasicEnergyEstimate}
\| u \|_0^2 + \| \nx u\|_1^2 \le \| u_f \|_0^2\,.
\ee
With $W = \sqrt{1 + |\nx w_1|^2/w_0^2} $\,, Lemma \ref{lem-PMi} leads to the improved $L^2$-estimate
\[
\label{energy-improved}
\| u\,W\|_0 \le \frac 1\kappa\,\| u_f \|_0\,.
\]
Lemma \ref{lem-PMi} can also be used to get an improved $H^1$-estimate.
\begin{lemma} \label{lem-energy-improved}
Let \eqref{Poincare1}, \eqref{V-ass1}, \eqref{V-ass2} and \eqref{H7} hold. Then any solution of \eqref{u-equ1} such that $\intRd u\,\rhoF\,dx=0$ satisfies
\[
\|W\,\nx u \|_1 \le C\,\| u_f \|_0\,.
\]
\end{lemma}
\begin{proof}
Multiplication of \eqref{u-equ1} by $u\,W^2$ and integration gives
\be{energy-inequality}
\| u\,W\|_0^2 + \|W\,\nx u \|_1^2 \le \|u\,W^2\|_0\,\|u_f\|_0 - \intRd w_1^2\,u\,\nx u\cdot \nx (W^2)\,dx\,.
\ee
Since $W^2 \le 1 + W\,|\nx w_1|/w_0$, \eqref{Poincare1} and Lemma~\ref{lem-PMhati} imply
\be{more-Poincare}
\|u\,W^2\|_0^2 \le 2\,(1/\lambda_M+1/\kappa')\,\|W\,\nx u\|_1^2\,.
\ee
The integrand in the last term above can be estimated by
\[
w_0\,|u|\,w_1\,|\nx u|\,\frac{2\,w_1\,W\,|\nx W|}{w_0} \le 2\,\sqrt 2\,c_3\,w_0\,|u|\,w_1\,|\nx u|\,W^2
\]
with the help of \eqref{V-ass2}, so that the integral is bounded by
\[
2\,\sqrt 2\,c_3\,\|u\,W\|_0\,\|W\,\nx u\|_1 \le 2\,\sqrt 2\,\frac{c_3}\kappa\,\|u_f\|_0\,\|W\,\nx u\|_1\,.
\]
Combining our results gives
\[
\| W\,\nx u \|_1^2 \le C\,\| u_f \|_0\,\| W\,\nx u \|_1
\]
with $C=\sqrt{2\,(1/\lambda_M+1/\kappa')}+2\,\sqrt 2\,\frac{c_3}\kappa$, thus completing the proof. \end{proof}

\begin{proof}[Proof of Proposition \ref{Lem:BK}] We follow the standard procedure for proving $H^2$-regularity of the solutions of second order elliptic equations with $L^2$ right hand sides: multiply \eqref{u-equ1} with $\nx\cdot(\nx u\,w_1^2/w_0^2)$, and integrate by parts twice. We also use the consequence $w_2 = w_1^2/w_0$ of the relations between the weights:
\begin{multline*}
\|\nx^2 u\|_2^2 = \intRd w_0\,(u-u_f)\,w_2\,\Delta_x u\,dx + \intRd w_0^2\,(u-u_f)\,\nx u \cdot \nx \left(\frac{w_1^2}{w_0^2}\right)\,dx \\
- \intRd w_1^2\,\nx u\,\nx^2 u\,\nx \left(\frac{w_1^2}{w_0^2}\right)\,dx - \intRd \frac{w_1^2}{w_0^2}\,\nx u\,\nx^2 u\,\nx (w_1^2)\,dx \\
- \intRd \left[ \nx u \cdot \nx \left( \frac{w_1^2}{w_0^2} \right) \right] \left[ \nx u \cdot \nx \left( w_1^2 \right) \right] dx\\
= I_1 + I_2 + I_3 + I_4 + I_5\,.
\end{multline*}
The first integral is easily estimated:
\[
|I_1| \le \big(\|u\|_0 + \|u_f\|_0 \big)\,\|\nx^2 u\|_2 \le 2\,\|u_f\|_0\,\|\nx^2 u\|_2\,,
\]
by \eqref{BasicEnergyEstimate}. For the second integral we use \eqref{V-ass3} and Lemma \ref{lem-energy-improved}:
\begin{eqnarray*}
|I_2| &\le& 2\,c_4 \intRd w_0\,(u-u_f)\,w_1\,|\nx u|\,\frac{|\nx w_1|}{w_0}\,dx \\
&\le& 2\,c_4\,\|u-u_f\|_0\,\|W\,\nx u\|_1 \le 4\,c_4\,C\,\|u_f\|_0^2\,.
\end{eqnarray*}
With the third, fourth and fifth integrals we proceed similarly:
\begin{eqnarray*}
|I_3| &\le& 2\,c_4 \intRd w_1\,|\nx u|\,w_2\,|\nx^2 u|\,\frac{|\nx w_1|}{w_0}\,dx \\
&\le& 2\,c_4\,\|W\,\nx u\|_1\,\|\nx^2 u\|_2 \le 2\,c_4\,C\,\| u_f\|_0\,\|\nx^2 u\|_2\,,\\
|I_4| &\le& 2 \intRd w_1\,|\nx u|\,w_2\,|\nx^2 u|\,\frac{|\nx w_1|}{w_0}\,dx \le 2\,c_4\,C\,\| u_f\|_0\,\|\nx^2 u\|_2\,, \\
|I_5| &\le& 4 \intRd w_1^2\,|\nx u|^2\,\big| \nx \big( \frac{w_1}{w_0} \big)\big|\,\frac{|\nx w_1|}{w_0}\,dx \\
&& \le 4\,c_4 \intRd w_1^2\,|\nx u|^2\,\frac{|\nx w_1|^2}{w_0^2}\,dx \le 4\,c_4\,\| W\,\nx u \|_1^2 \le 4\,c_4\,C\,\| u_f \|_0^2\,.
\end{eqnarray*}
The combination of our results gives
\[
\|\nx^2 u\|_2^2 \le K\,\|u_f\|_0\,\Big( \|u_f\|_0 + \|\nx^2 u\|_2 \Big)
\]
for some explicit constant $K>0$, which completes the proof. \end{proof}

\section{Maxwellian equilibria}

When the local equilibrium is a Maxwellian distribution, the global equilibrium has the form
\be{Maxwell}
F(x,v) = \rhoF(V(x))\,\BGprofile(v)\,, \quad\mbox{with}\; \rhoF(V) = e^{-V}\quad\mbox{and}\;\BGprofile(v) = \frac{e^{-|v|^2/2}}{(2\,\pi)^{d/2}}\,.
\ee
In this framework, Assumption {\bf (H0)} is a consequence of
\par\medskip\noindent {\bf Assumption (H0.1)} The external potential $V\in C^2(\R^d)$ is such that $e^{-V}\in L^1(dx)$. \par\medskip
As far as the macroscopic coercivity condition {\bf (H2)} and the boundedness of $\A\,\T\,(1-\Pi)$, i.e. the first part of {\bf (H4)}, are concerned, no further details of the collision operator are required. Consider first the issue of equivalent conditions for {\bf (H2)}.

With $w:=\rho\,e^{V/2}$ the macroscopic coercivity condition \eqref{eq:macroSG} is equivalent to
\[
\intRd \Big[|\nx w|^2+\big(\tfrac 14\,|\nx V|^2-\tfrac 12\,\Delta V\big) w^2\,\Big] dx \ge \lambda_M\,\intRd w^2\,dx\,,
\]
under the orthogonality condition $\intRd w\,e^{-V/2}\,dx=0$. The first eigenvalue of the Schr\"odinger operator $\mathcal H:=-\Delta +\frac 14\,|\nx V|^2-\frac 12\,\Delta V$ is zero. It is non-degenerate, and the corresponding eigenfunction is $w=e^{-V/2}$. According to \cite{MR0133586}, Inequality~\eqref{eq:macroSG} holds if and only if the lower end of the continuous spectrum of $\mathcal H$ is positive, that~is
\par\medskip\noindent {\bf Assumption (H2.1)} \quad$\liminf_{|x|\to\infty}\left(|\nx V|^2-2\,\Delta V\right)>0$.\par\medskip
As a consequence, macroscopic coercivity holds if $\Delta V$ is negligible compared to $|\nx V|^2$ as $|x|\to\infty$, and if $\liminf_{|x|\to\infty}|\nx V|>0$. An example of such a potential is $V(x)=(1+|x|^2)^\beta$ for some $\beta\geq 1/2$. See for instance \cite[A.19. Some criteria for Poincar\'e inequalities, page. 137]{Mem-villani} for an elementary proof if $\lim_{|x|\to\infty}\left(|\nx V|^2-2\,\Delta V\right)=\infty$, and \cite{MR2386063} for some recent considerations on Poincar\'e inequalities when $e^{-V}$ is a probability measure.

Since all three weights $\rhoF$, $\mF$, and $M_F$ are constant multiples of $e^{-V}$, the framework of Section \ref{sec:regularity} can be used for the boundedness
of $\A\,\T\,(1-\Pi)$. Assumptions \eqref{V-ass1}, \eqref{V-ass2} are
satisfied if
\par\medskip\noindent {\bf Assumption (H4.1)} There exist constants $c_1 > 0$, $c_2\in[0,1)$, and $c_3>0$, such that
\[
\Delta_x V \le c_1 + \frac{c_2}{2}\,|\nx V|^2\,,\quad |\nx^2 V| \le c_3 \left(1 + |\nx V| \right)\,.
\]
Assumption \eqref{V-ass3} holds trivially (since $w_1/w_0=const$), and \eqref{H7} can be translated~Êto
\[
\intRd |\nx V|^2 e^{-V}\,dx < \infty\,,
\]
which follows from {\bf (H0.1)} and {\bf (H4.1)} by
\begin{eqnarray*}
\intRd |\nx V|^2 e^{-V}\,dx &=& -\intRd \nx V\cdot\nx e^{-V}\,dx = \intRd \Delta_x V\,e^{-V}\,dx \\
&\le& c_1 \intRd e^{-V}\,dx + \frac{c_2}{2} \intRd |\nx V|^2 e^{-V}\,dx\,.
\end{eqnarray*}

\subsection{BGK operator}
\label{BGK-Maxwell}

For the BGK collision operator
\[
\L\,= \Pi - 1\,,
\]
the microscopic coercivity condition {\bf (H1)} is trivially satisfied with $\lambda_m = 1$, and, since $\L$ is bounded (by 1), the boundedness of $\A\,\L$ follows from Lemma \ref{lemma1}.
\begin{theorem}\label{Thm:HerauImproved}
Let $\L=\Pi-1$, and let the external potential satisfy {\bf (H0.1)}, {\bf (H2.1)}, and {\bf (H4.1)}. Then solutions of \eqref{eq:base} with initial data in $L^2(d\mu)$ decay exponentially to the global equilibrium given by \eqref{Maxwell}.
\end{theorem}
This result is an improvement upon the work of H\'erau \cite{Herau}, since the requirements for the external potential are weaker. In particular, H\'erau's result requires potentials with at most quadratic growth at infinity, whereas an arbitrary superlinear growth is permitted by {\bf (H2.1)}, {\bf (H4.1)}.

\subsection{Fokker-Planck operator} For the Fokker-Planck collision operator
\[
\L\,f = \nv\cdot(\nv f + v\,f)\,,
\]
the microscopic coercivity condition {\bf (H1)} is equivalent to the Poincar\'e inequality for the Gaussian measure $\BGprofile(v)\,dv$, which satisfies
\begin{eqnarray*}
-\scalar{\L\,f}{f} &=& \intRd e^V \intRd \BGprofile \left| \nv \tfrac{(1-\Pi)\,f}{\BGprofile} \right|^2 dv\,dx \\
&\ge& \lambda_m \intRd e^V \intRd \BGprofile \left|\tfrac{(1-\Pi)\,f}{\BGprofile}\right|^2 dv\,dx = \lambda_m\,\|(1-\Pi)\,f\|^2\,.
\end{eqnarray*}
A somewhat surprising fact is the boundedness of $\A\,\L$, although $\L$ is an unbounded operator. Since $\A = -(1+(\T\,\Pi)^*(\T\,\Pi))^{-1} \Pi\,\T$ and $\Pi\,\T\,f = \frac{F}{\rhoF}\nx\cdot j_f$ where $j_f$ is the flux given by $j_f = \intRd v\,f\,dv$, the identity $j_{\L\,f} = -j_f$ implies $\A\,\L = -\A$, and the boundedness of $\A\,\L$ is a consequence of Lemma \ref{lemma1}.
\begin{theorem}
Let $\L\,f= \nv\cdot(\nv f + v\,f)$ and assume that the external potential satisfies {\bf (H0.1)}, {\bf (H2.1)}, and {\bf (H4.1)}. Then solutions of \eqref{eq:base} with initial data in $L^2(d\mu)$ decay exponentially to the global equilibrium given by \eqref{Maxwell}.
\end{theorem}
The above assumptions are similar to those of \cite{Mem-villani}, which are weaker than those of \cite{Herau-Nier}. Moreover this result is an important improvement compared to \cite{Mem-villani} as involves an $L^2$ setting rather than a $H^1$ setting. Let us emphasize that the latter point is not a technical issue and answers an open question raised in \cite[Part~II, Section~13]{Mem-villani}).

\subsection{Scattering operators (without detailed balance)}

Consider a scattering operator that can be written as
\be{scattering}
(\L\,f)(v) = \intRd [k(v^*\to v)f(v^*) - k(v\to v^*)f(v)]\,dv^*\,, \quad k\ge 0
\ee
where $k(v^*\to v)$ denotes the transition probability of changing the velocity $v^*$ into $v$. Such an equation obviously conserves mass. Rotational symmetry can be enforced as a consequence of the assumption $k(Rv^*\to Rv) = k(v^*\to v)$, for all~$v$,~$v^*$, and for all rotation matrices $R$. \emph{Detailed balance} would mean that the integrand vanishes, whenever $f$ is a local equilibrium distribution. We shall only require that $\L\,\BGprofile = 0$. In the right hand side $f$ can be replaced by $(1-\Pi)\,f$. It has been shown in \cite{MR1803225} that in this case an H-theorem holds:
\[
-\intRd \L\,f \frac{f}{\BGprofile}\,dv = \frac 14\,\intRd \left( \frac{k(v^*\to v)}{\BGprofile} + \frac{k(v\to v^*)}{\BGprofile^*}\right)\BGprofile\,\BGprofile^* \left( \frac{f}{\BGprofile} - \frac{f^*}{\BGprofile^*}\right)^2\,dv^*\,dv
\,,
\]
where $f^*$ denotes $f(v^*)$. A sufficient condition for microscopic coercivity is
\par\medskip\noindent{\bf Assumption (H1.1)}\quad$\displaystyle{\frac{k(v^*\to v)}{\BGprofile} + \frac{k(v\to v^*)}{\BGprofile^*} \ge 2\lambda_m > 0}$.
\par\medskip Note that, because of $\L\,\BGprofile = 0$, the collision frequency $\nu(v) = \intRd k(v\to v^*)\,dv^*$ can be written as
\[
\nu(v) = \frac{1}{2} \intRd \left( \frac{k(v^*\to v)}{\BGprofile} + \frac{k(v\to v^*)}{\BGprofile^*}\right) \BGprofile^*\,dv^*\,.
\]
Thus, {\bf (H1.1)} implies $\nu(v)\ge \lambda_m$.

For proving the boundedness of $\A\,\L$, note that $g=\A\,\L\,f$ can be written as $g=u\,F$ with
\[
u\,e^{-V}-\nx\cdot\left(e^{-V}\nx u\right)=-\nx\cdot j_{\L\,f}\,.
\]
Multiplication by $u$ and integration gives
\[
\intRd u^2e^{-V}\,dx+\intRd |\nx u|^2\,e^{-V}\,dx=\intRd\nx u\cdot j_{\L\,f}\,dx\,.
\]
For the gain term $\L^+ f = \intRd k(v^*\to v)f^*\,dv^*$, by applying the Cauchy-Schwarz inequality twice, we get
\begin{eqnarray*}
|j_{\L^+ f}| &\le& \intRd |v| \left[\,{\intRd \frac{|f^*|^2}{\BGprofile^*}\,dv^* \intRd k(v^*\to v)^2 \BGprofile^*\,dv^*}\right]^\frac 12dv \\
&\le& \left[\,\intRd \frac{|f^*|^2}{\BGprofile^*}\,dv^* \iint_{\R^d\times\R^d} |v|^2\,k(v^*\to v)^2 \frac{\BGprofile^*}{\BGprofile}\,dv^*\,dv\right]^\frac 12.
\end{eqnarray*}
An analogous estimate for the loss term holds, so that we finally have
\[
\left|\intRd\nx u\cdot j_{\L\,f}\,dx\right| \le C \left(\intRd |\nx u|^2\,e^{-V}\,dx\right)^{1/2}\,\nrm f
\]
under the assumption
\par\medskip\noindent {\bf Assumption (H4.2)} \quad$\displaystyle{\iint_{\R^d\times\R^d} \left(|v|^2 + |v^*|^2\right) \,k\left(v^*\to v\right)^2\,\frac{\BGprofile^*}{\BGprofile} \,dv^*\,dv < \infty}$\,. \par\medskip
\noindent As a consequence, $\nrm{\A\,\L\,f}=\left(\intRd |u|^2\,e^{-V}\,dx\right)^{1/2 }\le C\,\nrm{f}$ and $\A\,\L$ is bounded. Combined with {\bf (H4.1)}, {\bf (H4.2)} shows that {\bf (H4)} holds.
\begin{theorem} Let $\L$ be given by \eqref{scattering}. If {\bf (H0.1)}, {\bf (H1.1)}, {\bf (H2.1)}, {\bf (H4.1)} and {\bf (H4.2)} hold, then solutions of \eqref{eq:base} with initial data in $L^2(d\mu)$ decay exponentially to the global equilibrium given by \eqref{Maxwell}.
\end{theorem}

\section{Linearized BGK operators}

\subsection{Motivation: nonlinear models}

Our motivation in this section comes from nonlinear BGK models with collision operators of the form
\[
Q(f) = \gamma\left( \frac 12\,|v|^2 - \omu(\rho_f)\right) - f\,.
\]
The operator is determined by the energy profile $\gamma(E)\ge 0$ which is assumed to be monotone decaying on $\gamma^{-1}(0,\infty)$. The (strictly increasing) function $\omu(\rho)$ is defined implicitly by the requirement of local mass conservation, i.e.
\[
\intRd \gamma\left( \frac 12\,|v|^2 - \omu(\rho) \right)dv = \rho\,.
\]
Global equilibria of the nonlinear equation $\partial_t f + \T\,f = Q(f)$ are given by $f_\infty(x,v) = \gamma(E(x,v) - \mu_\infty)$, where the constant $\mu_\infty$ is determined by the total mass, and the macroscopic equilibrium density $\rho_\infty$ by $\mu_\infty -V(x) = \omu(\rho_\infty(x))$. In this section, we shall investigate the linearized stability of these equilibria, leading to the linear equation \eqref{eq:base} with the linearized collision operator $\L = \Pi-1$ with
\[
F(x,v) = -\,\gamma'(E(x,v) - \mu_\infty)\,,\quad \rhoF(x) = \frac{1}{\omu'(\rho_\infty(x))}\,.
\]
Note that in the Maxwellian case $\gamma(E) = e^{-E}$, $\omu(\rho) = \log(\rho (2\,\pi)^{-d/2})$, the operator $Q$ is linear, and therefore equal to $\L$. This case has already been investigated in Section \ref{BGK-Maxwell}.

The macroscopic limit
\[
\partial_t \rho = \nx\cdot \big(\nx\nu(\rho) + \rho\,\nx V\big)\,,
\]
of the nonlinear equation is a drift-diffusion equation with nonlinear diffusivity $\diff(\rho) = \nu'(\rho) = \rho\,\omu'(\rho)$ (see \cite{DoMaOeSc} for a justification). Macroscopic limit and linearization commute in the sense that the linearization
\[
\partial_t \rho = \nx\cdot \big(\nx(\sigma(\rho_\infty)\,\rho) + \rho\,\nx V\big)
\]
of the macroscopic equation is the macroscopic limit of the linearized kinetic equation.

In the following section we consider a family of equilibrium energy distributions~$\gamma$, giving rise to nonlinear diffusions of fast diffusion type. As in Section \ref{BGK-Maxwell}, boundedness of $\L$ and microscopic coercivity are straightforward. Since $j_{\L f}=-j_f$, $\A\,\L$ is bounded, and $\L\,\A=0$ is easy to check. It remains to check the macroscopic coercivity condition and the boundedness of $\A\,\T\,(1-\Pi)$, corresponding to {\bf (H2)} and {\bf (H4)} respectively.

\subsection{Fast diffusion} The choice $\gamma(E) = E^{-d/2 - 1/(1-m)}$ and $m<1$ leads to
\[
\nu(\rho) = c\left\{ \begin{array}{ll} \rho^m & \mbox{for } m\ne 0\,,\\ \log\rho & \mbox{for } m=0\,, \end{array}\right.
\]
with a constant $c>0$ depending on $m$ and $d$. For $m<1$, we also compute the moments
\begin{eqnarray*} \label{FD-ass}
\rhoF = c_0 (V-\mu_\infty)^{-1-1/(1-m)}\,,\;&& \mF = c_1 (V-\mu_\infty)^{-1/(1-m)}\,,\\
&&M_F = c_2 (V-\mu_\infty)^{1-1/(1-m)}\,,\nonumber
\end{eqnarray*}
where the positive constants $c_0$, $c_1$, $c_2$ depend on $m$ and $d$. For the external potential we shall, for notational convenience, only consider the choice
\be{FD-V}
V(x)-\mu_\infty = (1+|x|^2)^\beta\,,\quad \beta > 0\,.
\ee
However, all our results are easily extendable to potentials whose asymptotic behaviour as $|x|\to\infty$ is given by \eqref{FD-V}. With these choices,
\par\medskip\noindent{\bf Assumption (H0.2)}\quad $\displaystyle{ \beta > \frac{d\,(1-m)}{2\,(2-m)}}$. \par\medskip
\noindent is necessary and sufficient for $\rhoF \in L^1(dx)$.

Macroscopic coercivity is related to Hardy-Poincar\'e inequalities. In \cite{BlBoDoGrVa,BlBoDoGrVa08,BDGV}, for any $d\ge 3$, $\alpha\neq \alpha^* := -(d-2)/2$ a positive constant $\mathcal C_{\alpha,d}$ is given explicitly, such that
\be{HardyPoincare}
\intRd |\nx u|^2\,(1+|x|^2)^{\alpha}\,dx \ge \mathcal C_{\alpha,d} \intRd u^2\,(1+|x|^2)^{\alpha-1}\,dx\,,
\ee
for all $u\in H^1\big((1+|x|^2 \big)^\alpha\,dx)$, under the additional condition $\intRd u\,(1+|x|^2)^{\alpha-1}\,dx=0$ if $\alpha<\alpha^*$, in which case the measure $(1+|x|^2)^{\alpha-1}\,dx$ is bounded. The Hardy-Poincar\'e inequality is equivalent to macroscopic coercivity {\bf (H2)} for $\beta=1$. A small generalization is even more useful for our purposes:
\begin{corollary}\label{Cor:BBDGV} Let $d\ge 3$, $\alpha_1 \ge \alpha_2 + 1$, and $\alpha_1 \ne \alpha^* := 1-d/2$ if $\alpha_1 = \alpha_2 + 1$. Let $w$ be a function such that $0 \le w(x) \le c\,(1+|x|^2)^{\alpha_2}$ for any $x\in\R^d$, for some $c>0$. Then there exists a positive constant $\mathcal K_{\alpha_1,\alpha_2,d}$ such that
\[
\label{HP} \intRd |\nx u|^2\,(1+|x|^2)^{\alpha_1}\,dx \ge \mathcal K_{\alpha_1,\alpha_2,d} \intRd u^2\,w\,dx
\]
for any $u$ such that $\intRd u\,w\,dx = 0$ if $\intRd w\,dx < \infty$. \end{corollary}
\begin{proof} The assumptions on $\alpha_1$ and $\alpha_2$ allow to choose $\alpha\ne\alpha^*$ with $\alpha_2+1\le \alpha \le \alpha_1$. Then Theorem 1 in \cite{BlBoDoGrVa} implies
\begin{eqnarray*}
&& \intRd |\nx u|^2(1+|x|^2)^{\alpha_1}\,dx \ge\intRd |\nx u|^2(1+|x|^2)^\alpha\,dx \\
&&\ge \mathcal C_{\alpha,d} \intRd |u^\bot|^2 (1+|x|^2)^{\alpha-1}\,dx \ge \mathcal C_{\alpha,d} \intRd |u^\bot|^2\,w\,dx\,,
\end{eqnarray*}
where $u = \bar u + u^\bot$ with $\bar u =0$ for $\alpha>\alpha^*$, and $\bar u = Const$ and $\intRd u^\bot(1+|x|^2)^{\alpha-1}dx=0$ for $\alpha<\alpha^*$. This completes the proof for $\alpha>\alpha^*$. Otherwise, $\intRd w\,dx<\infty$ holds, and the right hand side can be estimated as follows:
\begin{eqnarray*}
\intRd |u-\bar u|^2\,w\,dx \ge \inf_{\mu\in\R}\intRd |u-\mu|^2\,w\,dx = \intRd u^2\,w\,dx\,,
\end{eqnarray*}
using the side condition $\intRd u\,w\,dx = 0$.
\end{proof}

To get examples, where macroscopic coercivity holds in the fast diffusion case, we apply Corollary~\ref{Cor:BBDGV} with $\alpha_1=-\beta/(1-m)$ and $\alpha_2=-\beta\,(2-m)/(1-m)$. Then
\par\medskip\noindent{\bf Assumption (H2.2)} \quad $d\ge 3$, $\beta \ge 1$, and $m\neq (d-4)/(d-2)$\par\medskip
\noindent implies {\bf (H2)}. Note that the last condition is needed only for $\beta=1$. It will however be useful in the following. For proving the boundedness of $\A\,\T\,(1-\Pi)$, a modified version of the framework of Section \ref{sec:regularity} can be used. Redefining
\[
W := (1+|x|^2)^{(\beta - 1)/2}\,,
\]
$\frac{\nx w_1}{w_0}\le c\,W$ holds with the notation of Section \ref{sec:regularity}. The result of Lemma \ref{lem-PMhati}:
\be{HP-improved}
\| W\,\nx u \|_1^2 \ge \kappa'\,\Big\| W^2\,u \Big\|_0^2\,,
\ee
is a direct consequence of the Hardy-Poincar\'e inequality with $\alpha = \beta - 1 + \beta/(m-1)$. Note that we have to require that $\alpha$ is different from $\alpha^*$. The proof of Lemma \ref{lem-energy-improved} uses Assumption \eqref{V-ass2}, which would together with {\bf (H2.2)} require $\beta=1$ and therefore by {\bf (H0.2)} $m>(d-4)/(d-2)$. This can be slightly improved by redoing the proof of Lemma \ref{lem-energy-improved}.
\begin{lemma} \label{lem-FD-ellipticity} Let {\bf (H2.2)} hold. Then there exists a constant $\beta_0 > 1$, depending on $m$ and $d$, such that for $\beta<\beta_0$ the operator $\A\,\T\,(1-\Pi)$ is bounded. \end{lemma}
\begin{proof} As mentioned above, \eqref{HP-improved} follows from \cite[Theorem 1]{BlBoDoGrVa}. Since, by {\bf (H2.2)}, $1/(m-1)\ne\alpha^*$, also $\beta-1 + \beta/(m-1)\ne\alpha^*$, if $\beta$ is close enough to 1. According to \cite{BDGV}, the explicit expression of the constant $\mathcal C_{\alpha,d}$ in \eqref{HardyPoincare} is a positive, continuous function of~$\alpha$ for $\alpha<\alpha^*$ and $\alpha>\alpha^*$ and $\kappa'= \kappa'(\beta,m,d)>0$ in \eqref{HP-improved} can be chosen to be continuous with respect to $\beta$ at $\beta=1$.

As in the proof of Lemma \ref{lem-energy-improved}, we derive the inequality \eqref{energy-inequality} and use
\[
w_1^2\,|\nx (W^2)| \le 2\,(\beta-1)\,w_0\,w_1\,W^3
\]
to estimate the last term by $2\,(\beta-1)\,\| u\,W^2\|_0\,\|W\,\nx u\|_1$. With the help of \eqref{HP-improved}, \eqref{energy-inequality} implies
\[
\|W\,\nx u\|_1^2 \le \frac{1}{\kappa'}\,\|W\,\nx u\|_1 \Big( \|u_f\|_0 + 2\,(\beta-1)\,\|W\,\nx u\|_1 \Big)\,.
\]
By the continuity of $\kappa'$, there exists $\beta_0(m,d) > 1$ such that $2\,(\beta-1)/\kappa'<1$ for $1\le\beta< \beta_0$. or such a $\beta$ the result of Lemma \ref{lem-energy-improved} follows. This allows to carry out the proof of Proposition~\ref{Lem:BK}, since Assumption \eqref{V-ass3}, which is used there, is satisfied. \end{proof}

As a consequence of this result, we formulate
\par\medskip\noindent{\bf Assumption (H4.3)} \quad $\beta<\beta_0$ with $\beta_0(m,d)$ from Lemma \ref{lem-FD-ellipticity}.\par\medskip
\begin{theorem}\label{Thm:NLBGK} With the above notations, let $\L = \Pi-1$, and assume that {\bf (H0.2)}, {\bf (H2.2)}, and {\bf (H4.3)} hold. Then solutions of \eqref{eq:base} with initial data in $L^2(d\mu)$ decay exponentially to the global equilibrium given by
\[
F(x,v) = \left( \frac 12\,|v|^2 + V(x)\right)^{-d/2 - 1/(1-m) - 1} \,,\quad V(x) = (1+|x|^2)^\beta\,,
\]
\end{theorem}
This result is, to our knowledge, the first hypocoercivity result for kinetic equation whose Gibbs state does not separate position and velocity variables.

\medskip\par\noindent\emph{Acknowledgements.\/} This work has been partially supported by the French-Austrian Amadeus project no. 13785UA, the ANR projects EVOL and CBDif-Fr, the Austrian Science Fund (project no. W8) and the European network DEASE. C.M. thanks the Cambridge University for hospitality and acknowledges support from Award No. KUK-I1-007-43, funded by the King Abdullah University of Science and Technology (KAUST).

\par\noindent{\scriptsize\copyright\ 2010 by the authors. This paper may be reproduced, in its entirety, for non-commercial purposes.}

\bibliographystyle{amsplain}\bibliography{References}

\providecommand{\bysame}{\leavevmode\hbox to3em{\hrulefill}\thinspace}
\providecommand{\MR}{\relax\ifhmode\unskip\space\fi MR }
\providecommand{\MRhref}[2]{%
  \href{http://www.ams.org/mathscinet-getitem?mr=#1}{#2}
}
\providecommand{\href}[2]{#2}
\begin{thebibliography}{10}

\bibitem{MR2386063}
Dominique Bakry, Franck Barthe, Patrick Cattiaux, and Arnaud Guillin, \emph{A
  simple proof of the {P}oincar\'e inequality for a large class of probability
  measures including the log-concave case}, Electron. Commun. Probab.
  \textbf{13} (2008), 60--66. \MR{MR2386063}

\bibitem{BlBoDoGrVa}
Adrien Blanchet, Matteo Bonforte, Jean Dolbeault, Gabriele Grillo, and
  Juan-Luis V{\'a}zquez, \emph{{H}ardy-{P}oincar{\'e} inequalities and
  applications to nonlinear diffusions}, Comptes Rendus Math{\'e}matique
  \textbf{344} (2007), no.~7, 431--436.

\bibitem{BlBoDoGrVa08}
Adrien Blanchet, Matteo Bonforte, Jean Dolbeault, Gabriele Grillo, and
  Juan-Luis V\'azquez, \emph{Asymptotics of the fast diffusion equation via
  entropy estimates}, Archive for Rational Mechanics and Analysis \textbf{191}
  (2009), no.~2, 347--385.

\bibitem{BDGV}
{M}atteo {B}onforte, {J}ean {D}olbeault, {G}abriele {G}rillo, and {J}uan-{L}uis
  {V}{\'a}zquez, \emph{{S}harp rates of decay of solutions to the nonlinear
  fast diffusion equation via functional inequalities}, Preprint hal-00404718,
  2009.

\bibitem{Caceres-Carrillo-Goudon}
Maria~J. C{\'a}ceres, Jos{\'e}~A. Carrillo, and Thierry Goudon,
  \emph{Equilibration rate for the linear inhomogeneous relaxation-time
  {B}oltzmann equation for charged particles}, Comm. Partial Differential
  Equations \textbf{28} (2003), no.~5-6, 969--989. \MR{MR1986057 (2004g:82111)}

\bibitem{MR0032898}
Carlo Cattaneo, \emph{Sulla conduzione del calore}, Atti Sem. Mat. Fis. Univ.
  Modena \textbf{3} (1949), 83--101. \MR{MR0032898 (11,362d)}

\bibitem{MR1803225}
P.~Degond, T.~Goudon, and F.~Poupaud, \emph{Diffusion limit for nonhomogeneous
  and non-micro-reversible processes}, Indiana Univ. Math. J. \textbf{49}
  (2000), no.~3, 1175--1198. \MR{MR1803225 (2002a:35012)}

\bibitem{Desvillettes-Villani-2001}
L.~Desvillettes and C.~Villani, \emph{On the trend to global equilibrium in
  spatially inhomogeneous entropy-dissipating systems: the linear
  {F}okker-{P}lanck equation}, Comm. Pure Appl. Math. \textbf{54} (2001),
  no.~1, 1--42. \MR{MR1787105 (2001h:82079)}

\bibitem{Desvillettes-Villani-2005}
\bysame, \emph{On the trend to global equilibrium for spatially inhomogeneous
  kinetic systems: the {B}oltzmann equation}, Invent. Math. \textbf{159}
  (2005), no.~2, 245--316. \MR{MR2116276 (2005j:82070)}

\bibitem{DoMaOeSc}
Jean Dolbeault, Peter Markowich, Dietmar {\"O}lz, and Christian Schmeiser,
  \emph{Non linear diffusions as limit of kinetic equations with relaxation
  collision kernels}, Arch. Ration. Mech. Anal. \textbf{186} (2007), no.~1,
  133--158. \MR{MR2338354}

\bibitem{Dolbeault2009511}
Jean Dolbeault, Cl{\'e}ment Mouhot, and Christian Schmeiser,
  \emph{Hypocoercivity for kinetic equations with linear relaxation terms},
  Comptes Rendus Mathematique \textbf{347} (2009), no.~9-10, 511 -- 516.

\bibitem{DMS-2part}
\bysame, \emph{Hypocoercivity for linear kinetic equations}, In preparation,
  2010.

\bibitem{Fellner-Neumann-Schmeiser}
Klemens Fellner, Lukas Neumann, and Christian Schmeiser, \emph{Convergence to
  global equilibrium for spatially inhomogeneous kinetic models of
  non-micro-reversible processes}, Monatsh. Math. \textbf{141} (2004), no.~4,
  289--299. \MR{MR2053654 (2005f:82116)}

\bibitem{Guo-2002-II}
Yan Guo, \emph{The {L}andau equation in a periodic box}, Comm. Math. Phys.
  \textbf{231} (2002), no.~3, 391--434. \MR{MR1946444 (2004c:82121)}

\bibitem{Guo-2002-I}
\bysame, \emph{The {V}lasov-{P}oisson-{B}oltzmann system near {M}axwellians},
  Comm. Pure Appl. Math. \textbf{55} (2002), no.~9, 1104--1135. \MR{MR1908664
  (2003b:82050)}

\bibitem{Guo-2003-I}
\bysame, \emph{Classical solutions to the {B}oltzmann equation for molecules
  with an angular cutoff}, Arch. Ration. Mech. Anal. \textbf{169} (2003),
  no.~4, 305--353. \MR{MR2013332 (2004i:82054)}

\bibitem{Guo-2003-II}
\bysame, \emph{The {V}lasov-{M}axwell-{B}oltzmann system near {M}axwellians},
  Invent. Math. \textbf{153} (2003), no.~3, 593--630. \MR{MR2000470
  (2004m:82123)}

\bibitem{Guo-2004}
\bysame, \emph{The {B}oltzmann equation in the whole space}, Indiana Univ.
  Math. J. \textbf{53} (2004), no.~4, 1081--1094. \MR{MR2095473 (2005g:35028)}

\bibitem{herau2010anisotropic}
F.~H{\'e}rau and K.~Pravda-Starov, \emph{Anisotropic hypoelliptic estimates for
  {L}andau-type operators}, Arxiv preprint arXiv:1003.3265, 2010.

\bibitem{Herau}
Fr{\'e}d{\'e}ric H{\'e}rau, \emph{Hypocoercivity and exponential time decay for
  the linear inhomogeneous relaxation {B}oltzmann equation}, Asymptot. Anal.
  \textbf{46} (2006), no.~3-4, 349--359. \MR{MR2215889 (2007b:35044)}

\bibitem{Herau-Nier}
Fr{\'e}d{\'e}ric H{\'e}rau and Francis Nier, \emph{Isotropic hypoellipticity
  and trend to equilibrium for the {F}okker-{P}lanck equation with a
  high-degree potential}, Arch. Ration. Mech. Anal. \textbf{171} (2004), no.~2,
  151--218. \MR{MR2034753 (2005f:82085)}

\bibitem{hitrik2009semiclassical}
M.~Hitrik and K.~Pravda-Starov, \emph{{Semiclassical hypoelliptic estimates for
  non-selfadjoint operators with double characteristics}}, To appear in
  Communications in Partial Differential Equations, 2010.

\bibitem{MR2274461}
Ming-Yi Lee, Tai-Ping Liu, and Shih-Hsien Yu, \emph{Large-time behavior of
  solutions for the {B}oltzmann equation with hard potentials}, Comm. Math.
  Phys. \textbf{269} (2007), no.~1, 17--37. \MR{MR2274461 (2007k:82115)}

\bibitem{MR2044894}
Tai-Ping Liu and Shih-Hsien Yu, \emph{Boltzmann equation: micro-macro
  decompositions and positivity of shock profiles}, Comm. Math. Phys.
  \textbf{246} (2004), no.~1, 133--179. \MR{MR2044894 (2005f:82101)}

\bibitem{MR2284213}
\bysame, \emph{Initial-boundary value problem for one-dimensional wave
  solutions of the {B}oltzmann equation}, Comm. Pure Appl. Math. \textbf{60}
  (2007), no.~3, 295--356. \MR{MR2284213}

\bibitem{Mouhot-Neumann}
Cl{\'e}ment Mouhot and Lukas Neumann, \emph{Quantitative perturbative study of
  convergence to equilibrium for collisional kinetic models in the torus},
  Nonlinearity \textbf{19} (2006), no.~4, 969--998. \MR{MR2214953
  (2007c:82032)}

\bibitem{Pazy-1983}
Ammon Pazy, \emph{Semigroups of linear operators and applications to partial
  differential equations}, Applied Mathematical Sciences, vol.~44,
  Springer-Verlag, New York, 1983. \MR{MR710486 (85g:47061)}

\bibitem{MR0133586}
Arne Persson, \emph{Bounds for the discrete part of the spectrum of a
  semi-bounded {S}chr\"odinger operator}, Math. Scand. \textbf{8} (1960),
  143--153. \MR{MR0133586 (24 \#A3412)}

\bibitem{Guo-Strain-2004}
Robert~M. Strain and Yan Guo, \emph{Stability of the relativistic {M}axwellian
  in a collisional plasma}, Comm. Math. Phys. \textbf{251} (2004), no.~2,
  263--320. \MR{MR2100057 (2005m:82155)}

\bibitem{Guo-Strain-2006}
\bysame, \emph{Almost exponential decay near {M}axwellian}, Comm. Partial
  Differential Equations \textbf{31} (2006), no.~1-3, 417--429. \MR{MR2209761
  (2006m:82042)}

\bibitem{Guo-Strain-2008}
\bysame, \emph{Exponential decay for soft potentials near {M}axwellian}, Arch.
  Ration. Mech. Anal. \textbf{187} (2008), no.~2, 287--339. \MR{MR2366140}

\bibitem{Ukai-1974}
Seiji Ukai, \emph{On the existence of global solutions of mixed problem for
  non-linear {B}oltzmann equation}, Proc. Japan Acad. \textbf{50} (1974),
  179--184. \MR{MR0363332 (50 \#15770)}

\bibitem{Mem-villani}
C{\'e}dric Villani, \emph{Hypocoercivity}, To appear in Memoirs Amer. Math.
  Soc., 2008.

\bibitem{MR2264611}
Shih-Hsien Yu, \emph{The development of the {G}reen's function for the
  {B}oltzmann equation}, J. Stat. Phys. \textbf{124} (2006), no.~2-4, 301--320.
  \MR{MR2264611 (2007i:82072)}

\end{thebibliography}
\end{document}